\documentclass[a4paper,12pt]{article}
\usepackage{amsmath}
\usepackage{amsthm}
\usepackage{graphicx}
\usepackage{verbatim}
\usepackage{epsfig}
\usepackage{hyperref}
\usepackage{float}
\usepackage{color}
\usepackage{fullpage}
\usepackage{amsfonts}
\usepackage{amscd}
\usepackage{amssymb}
\usepackage{graphics}
\usepackage{amsmath}
\usepackage[english]{babel}
\usepackage{bbm}
\usepackage{yfonts}
\usepackage{mathrsfs}
\usepackage{color}
\allowdisplaybreaks[4]
\DeclareFontFamily{OML}{rsfs}{\skewchar\font'177}
\DeclareFontShape{OML}{rsfs}{m}{n}{ <5> <6> rsfs5 <7> <8> <9>
rsfs7 <10> <10.95> <12> <14.4> <17.28> <20.74> <24.88> rsfs10 }{}
\DeclareMathAlphabet{\mathfs}{OML}{rsfs}{m}{n}

\newcommand{\BP}{{\mathbb{P}}}

\newcommand{\BZ}{{\mathbb{Z}}}

\newcommand{\CC}{{\mathcal{C}}}

\newcommand{\CL}{{\mathcal{M}}}

\newcommand{\CR}{{\mathcal{R}}}

\newcommand{\la}{\langle}
\newcommand{\ra}{\rangle}

\newcommand{\ind}{{\mathbbm{1}}}

\newcommand{\bae}{\begin{equation}\begin{aligned}}
\newcommand{\eae}{\end{aligned}\end{equation}}

\newcommand{\pr}{\mathbb{P}}
\newcommand{\Z}{\mathbb{Z}}

\newtheorem{thm}{Theorem}[section]
\newtheorem{prop}[thm]{Proposition}
\newtheorem{lem}[thm]{Lemma}

\newtheorem{definition}{Definition}[section]

\begin{document}
\numberwithin{equation}{section} \numberwithin{figure}{section}
\title{Geometry of the random interlacement}
\author{Eviatar B. Procaccia\footnote{Research supported by ISF grant 1300/08 and EU grant PIRG04-GA-2008-239317 }\footnote{Weizmann Institute of Science}, Johan Tykesson \footnotemark[\value{footnote}]}
\maketitle
\begin{abstract}
We consider the geometry of random interlacements on the
$d$-dimensional lattice. We use ideas from stochastic dimension
theory developed in \cite{benjamini2004geometry} to prove the
following: Given that two vertices $x,y$ belong to the
interlacement set, it is possible to find a path between $x$ and
$y$ contained in the trace left by at most $\lceil d/2  \rceil$
trajectories from the underlying Poisson point process. Moreover, this result
is sharp in the sense that there are pairs of points in the
interlacement set which cannot be connected by a path using the
traces of at most $\lceil d/2 \rceil-1$ trajectories.
\end{abstract}
\section{Introduction}
The model of random interlacements was introduced by Sznitman in
\cite{sznitmanvacant}, on the graph ${\mathbb Z}^d$, $d\ge 3$.
Informally, the random interlacement is the trace left by a
Poisson point process on the space of doubly infinite trajectories
on ${\mathbb Z}^d$. The intensity measure of the Poisson process
is given by $u\nu$, $u>0$ and $\nu$ is a measure on the space of
doubly infinite trajectories, see~\eqref{e:nudef} below. This is a
site percolation model that exhibits infinite-range dependence,
which for example presents serious complications when trying to
adapt techniques developed for standard independent site
percolation.

In \cite{sznitmanvacant}, it was proved that the random
interlacement on ${\mathbb Z}^d$ is always a connected set. In
this paper we prove a stronger statement (for precise formulation,
see Theorem \ref{thm:main}):
\newline

{\it Given that two vertices $x,y\in\Z^d$ belong to the interlacement
set, it is a.s. possible to find a path between $x$ and $y$
contained in the trace left by at most $\lceil d/2 \rceil$
trajectories from the underlying Poisson point process. Moreover,
this result is sharp in the sense that a.s. there are pairs of
points in the interlacement set which cannot be connected by a
path using the traces of at most $\lceil d/2 \rceil-1$
trajectories.}
\newline

Our method is based on the concept of stochastic dimension (see
Section~\ref{sec:stochdim} below) introduced by Benjamini, Kesten,
Peres and Schramm, \cite{benjamini2004geometry}. They studied the
geometry of the so called uniform spanning forest, and here we
show how their techniques can be adapted to the study of the
geometry of the random interlacements. 

In Section \ref{sec:randominterlace} we introduce the model of
random interlacements more precisely. In Section
\ref{sec:stochdim} we give the required background on stochastic
dimension and random relations from \cite{benjamini2004geometry}.
Finally the precise statement and proof of Theorem \ref{thm:main}
is split in two parts: the lower bound is given in Sections
\ref{sec:lower} and the upper bound in Section \ref{sec:upper}.

Throughout the paper, $c$ and $c'$ will denote dimension dependant constants, and their values may change from place to place. Dependence of additional parameters will be indicated, for example $c(u)$ will stand for a constant depending on $d$ and $u$.

During the last stages of this research we have learned that B.
Rath and A. Sapozhnikov, see \cite{ráth2010connectivity}, have
solved this problem independently. Their proof is significantly
different from the proof we present in this paper.

\section{Preliminaries}\label{s.preliminaries}
In this section we recall the model of random interlacements from
\cite{sznitmanvacant} and the concept of stochastic dimension from
\cite{benjamini2004geometry}.
\subsection{Random interlacements}\label{sec:randominterlace}
Let $W$ and $W_+$ be the spaces of doubly infinite and infinite trajectories in $\Z^d$
that spend only a finite amount of time in finite subsets of $\Z^d$:
$$W=\{\gamma\,:\,\Z\to \Z^d;\,|\gamma(n)-\gamma(n+1)|=1,\,\forall n\in
\Z;\,\lim_{n\to\pm\infty}|\gamma(n)|=\infty\},$$
$$W_+=\{\gamma\,:{\mathbb N}\to \Z^d;\,|\gamma(n)-\gamma(n+1)|=1,\,\forall n\in
\Z;\,\lim_{n\to\infty}|\gamma(n)|=\infty\}.$$ The canonical
coordinates on ${\mathcal W}$ and ${\mathcal W}_+$ will be denoted
by $X_n$, $n\in {\mathbb Z}$ and $X_n$, $n\in {\mathbb N}$
respectively. We endow $W$ and $W_+$ with the sigma-algebras
${\mathcal W}$ and ${\mathcal W}_+$, respectively which are
generated by the canonical coordinates. For $\gamma\in W$, let range$(\gamma)=\gamma({\mathbb Z})$. Furthermore, consider the
space $W^*$ of trajectories in $W$ modulo time shift:
$$W^*=W/\sim,\mbox{ where }w\sim w'\Longleftrightarrow
w(\cdot)=w'(\cdot+k)\mbox{ for some }k\in \Z.$$ Let $\pi^*$ be the
canonical projection from $W$ to $W^*$, and let ${\mathcal W}^*$
be the sigma-algebra on $W^*$ given by $\{A\subset
W^*\,:\,(\pi^*)^{-1}(A)\in {\mathcal W}\}.$ Given
$K\subset{\mathbb Z}^d$ and $\gamma\in W_+$, let
$\tilde{H}_K(\gamma)$ denote the hitting time of $K$ by $\gamma$:
\begin{equation}\label{hitting}
\tilde{H}_K(\gamma)=\inf\{n\ge 1\,:\,X_n(\gamma)\in K\}.
\end{equation}
For $x\in {\mathbb Z}^d$, let $P_x$ be the law on $(W_+,{\mathcal
W}_+)$ corresponding to simple random walk started at $x$, and for
$K\subset \BZ^d$, let $P_x^K$ be the law of simple random walk,
conditioned on not hitting $K$. Define the equilibrium measure of
$K$: \bae e_K(x)=\left\{\begin{array}{ll}
P_x[\tilde{H}_K=\infty],\quad &x\in K\\
0,\quad &x\notin K.\end{array}\right. \eae Define the capacity of
a set $K\subset \BZ^d$ as
\begin{equation}
\text{cap}(K)=\sum_{x\in \BZ^d} e_K(x).
\end{equation}
The normalized equilibrium measure of $K$ is defined as
\begin{align}
\tilde{e}_K(\cdot)=e_K(\cdot)/\text{cap}(K).
\end{align}
For $x,y\in {\mathbb Z}^d$ we let $|x-y|=\|x-y\|_1$
We will repeatedly make use of the following well-known estimates of hitting-probabilities. For any $x,y\in {\mathbb Z}^d$ with $|x-y|\ge 1$,
\begin{equation}\label{e.hitbounds}
c |x-y|^{-(d-2)}\le P_x[\tilde{H}_y<\infty]\le c' |x-y|^{-(d-2)},
\end{equation}
see for example Theorem 4.3.1 in \cite{lawler2010random}. Next we
define a Poisson point process on $W^*\times{\mathbb R}_+$. The
intensity measure of the Poisson point process is given by the
product of a certain measure $\nu$ and the Lebesque measure on
${\mathbb R}_+$. The measure $\nu$ was constructed by Sznitman in
\cite{sznitmanvacant}, and now we characterize it. For
$K\subset{\mathbb Z}^d$, let $W_K$ denote the set of trajectories
in $W$ that enter $K$. Let $W_K^{*}=\pi^*(W_K)$. Define $Q_K$ to
be the finite measure on $W_K$ such that for $A,B\in {\mathcal
W}_+$ and $x\in {\mathbb Z}^d$,
\begin{equation}\label{e:Qdef}Q_K[(X_{-n})_{n\ge 0}\in A,X_0=x,(X_n)_{n\ge 0}\in B]=P_x^K[A]e_K(x)P_x[B].\end{equation} The measure $\nu$ is the unique $\sigma$-finite measure such that
\begin{equation}\label{e:nudef}
\ind_{W^*_K}\nu=\pi^*\circ Q_K,\text{ }\forall K\subset{\mathbb Z}^d\text{ finite}.
\end{equation}
The existence and uniqueness of the measure was proved in Theorem 1.1 of \cite{sznitmanvacant}. Consider the set of point measures
in $W^{*}\times {\mathbb R}_+$:
\bae
\Omega=\bigg\{&\omega=\sum_{i=1}^{\infty}\delta_{(w_i^*,u_i)};\,w_i^*\in W^*, u_i\in {\mathbb R}_+,\,\\
&\omega(W_K^*\times [0,u])<\infty,\mbox{ for every finite}K\subset \Z^d\text{ and }u\in {\mathbb R}_+\bigg\},
\eae

where $W_K^*$ denotes the set of trajectories in $W^*$ that
intersect $K$. Also consider the space of point measures on $W^*$:
\begin{equation}\tilde{\Omega}=\left\{\sigma=\sum_{i=1}^{\infty}
\delta_{w_i^*};\,w_i^*\in W^*,\, \sigma(W_K^*)<\infty,\,\mbox{ for
every finite }K\subset \Z^d\right\}.
\end{equation} For $u>u'\ge 0$, we define the mapping $\omega_{u',u}$ from $\Omega$ into
$\tilde{\Omega}$ by
\begin{align}\label{e.omegaudef}\omega_{u',u}=\sum_{i=1}^\infty \delta_{w_i^*}{\ind}\{u'\le u_i\le u\},\text{ for }\omega=\sum_{i=1}^{\infty}
\delta_{(w_i^*,u_i)}\in \Omega.\end{align} If $u'=0$, we write
$\omega_u$.  Sometimes we will refer trajectories in $\omega_u$,
rather than in the support of $\omega_u$. On $\Omega$ we let
${\mathbb P}$ be the law of a Poisson point process with intensity
measure given by $\nu(dw^*)dx$. Observe that under ${\mathbb P}$,
the point process $\omega_{u,u'}$ is a Poisson point process on
$\tilde{\Omega}$ with intensity measure $(u-u') \nu(dw^*)$. Given
$\sigma\in \tilde{\Omega}$, we define
\begin{equation}
{\mathcal
I}(\sigma)=\bigcup_{w^*\in\text{supp}(\sigma)}w^*({\mathbb Z}).
\end{equation}
For $0\le u'\le u$, we define
\begin{equation}\label{e.ridef}
{\mathcal I}^{u',u}={\mathcal I}(\omega_{u',u}),
\end{equation}
which we call the \emph{random interlacement set} between levels
$u'$ and $u$. In case $u'=0$, we write ${\mathcal I}^u$.

We introduce a decomposition of $\omega_u$ as follows. Let
$\omega_u^0$ be the point measure supported on those $w_i^*\in
\text{supp}(\omega_u)$ for which $0\in w_i^*({\mathbb Z})$. Then
proceed inductively: given $\omega_u^0,...,\omega_u^{k-1}$, define
$\omega_u^{k}$ to be the point measure supported on those
$w_i^*\in\text{supp}(\omega_u)$ such that $w_i^*\notin
\mbox{supp}(\sum_{i=0}^{k-1}\omega_u^{i})$ and $w_i^*({\mathbb
Z})\cap\left( \cup_{w_i^*\in
\text{supp}(\omega_u^{k-1})}w_i({\mathbb Z})\right)\neq\emptyset.$

We define $\omega_{u|_A}$ to be $\omega_u$ restricted to $A\subset W^*$.

\subsection{Stochastic dimension}\label{sec:stochdim}
In this section, we recall some definitions and results from
\cite{benjamini2004geometry} and adapt them to our needs. For $x,y\in\Z^d$, let $\la
xy \ra= 1+|x-y|$. Suppose $W\subset {\mathbb Z}^d$ is finite and
that $\tau$ is a tree on $W$. Let $\la \tau \ra=\Pi_{\{x,y\}\in
\tau}\la xy\ra$. Finally let $\la W \ra = \min_{\tau} \la \tau
\ra$ where the minimum is taken over all trees on the vertex set
$W$. For example, for $n$ vertices $x_1,...,x_n$, $\la
x_1...x_n\ra$ is the minimum of $n^{n-2}$ products with $n-1$
factors in each.
\begin{definition}\label{def:stocdim}
Let ${\mathcal R}$ be a random subset of $ {\mathbb
Z}^d\times{\mathbb Z}^d$ with distribution ${\bf P}$. We will
think of ${\mathcal R}$ as a relation and for $(x,y)\in {\mathbb
Z}^d\times{\mathbb Z}^d$, we write $x{\mathcal R}y$ if
$(x,y)\in{\mathcal R}$. Let $\alpha\in[0,d)$. We say that
${\mathcal R}$ has \emph{stochastic dimension} $d-\alpha$ if there
exists a constant $c=c({\mathcal R})<\infty$ such that
\begin{equation}\label{e.sreldef1} c\,{\bf P}[x{\mathcal R}y]\ge
\la x y \ra^{-\alpha},
\end{equation}
and
\begin{equation}\label{e.sreldef2}
{\bf P}[x{\mathcal R}y,\,z{\mathcal R}v]\le c\la x
y\ra^{-\alpha}\la z  v\ra^{-\alpha}+c \la x y z v\ra^{-\alpha},
\end{equation}
for all $x,y,z,v\in{\mathbb Z}^d$.

If ${\mathcal R}$ has stochastic dimension $d-\alpha$, then we
write $dim_{S}({\mathcal R})=d-\alpha.$
\end{definition}
Observe that $\inf_{x,y\in {\mathbb Z}^d}{\bf P}[x{\mathcal
R}y]>0$ if and only if $dim_{S}({\mathcal R})=d$.
\begin{definition}
Let $\CR$ and $\CL$ be any two random relations. We define the
composition $\CR \CL$ to be the set of all $(x,z)\in \BZ^d\times
\BZ^d$ such that there exists some $y\in \BZ^d$ for which $x\CR y$
and $y\CL z$ holds. The $n$-fold composition of a relation
${\mathcal R}$ will be denoted by ${\mathcal R}^{(n)}$.
\end{definition}
Next we restate Theorem 2.4 of \cite{benjamini2004geometry}, which
we will use extensively.
\begin{thm}\label{thm:dimsum}
Let $\mathcal{L},\mathcal{R}\subset\Z^d$ be two independent random relations. Then
\[
dim_{S}(\mathcal{L}{\mathcal R})=\min\left\{dim_{S}({\mathcal L})+dim_{S}({\mathcal R}),d\right\},
\]
when $dim_{S}({\mathcal L})$ and $dim_{S}({\mathcal R})$ exist.
\end{thm}

For each $x \in {\mathbb Z}^d$, we choose a trajectory $w_x\in
W_+$ according to $P_x$. Also assume that $w_x$ and $w_y$ are
independent for $x\neq y$ and that the collection $(w_x)_{x\in
{\mathbb Z}}$ is independent of $\omega$.

We will take interest in several particular relations, defined in
terms of $\omega$ and the collection $(w_x)_{x\in {\mathbb Z}^d}$.
For $\omega=\sum_{i=1}^{\infty} \delta_{(w_i^*,u_i)}\in \Omega$,
$t_2\ge t_1\ge 0$, and $n\in {\mathbb N}$ let
\begin{enumerate}
\item  \bae\label{e.RIrelation}{\mathcal
M}_{t_1,t_2}&=\big\{(x,y)\in{\mathbb Z}^d\times{\mathbb
Z}^d\,:\,\exists \gamma\in
\text{supp}(\omega_{t_1,t_2})\,:\,x,y\in \gamma({\mathbb
Z})\big\}.\eae
If $t_1=0$, we will write ${\mathcal M}_{t_2}$ as shorthand for
${\mathcal M}_{t_1,t_2}$.
\item ${\mathcal L}=\{(x,y)\in {\mathbb Z}^d\times {\mathbb Z}^d\,:\,y\in \text{range}(w_x)\}$
\item ${\mathcal R}=\{(x,y)\in {\mathbb Z}^d\times {\mathbb Z}^d\,:\,x\in \text{range}(w_y)\}$
\item \begin{align}{\mathcal
C}_n={\mathcal L}\left(\prod_{i=2}^{n-1}{\mathcal
M}_{u(i-1)/n,ui/n}\right){\mathcal R},\,n\ge 3.
\end{align}
\end{enumerate}

\begin{thm}\label{thm:main}
For any $d\ge 3$ and all $x,y\in \Z^d$, \[\pr\left[x\mathcal{M}_u^{\lceil\frac{d}{2}\rceil}y|x,y\in\mathcal I^u\right]=1.\] In addition we have \[\pr\left[\exists x,y\in\mathcal{I}^u,y\notin\{z:x\mathcal{M}_u^{\lceil\frac{d}{2}\rceil-1}z\}\right]=1.\]
\end{thm}
For $d=3,4$ the theorem follows easily from the fact that two
independent simple random walk trajectories intersect each other
almost surely, and we omit the details of these two cases. From
now on, we will assume that $d\ge 5$. A key step in the proof of
our main theorem, is to show that for every $x,y\in {\mathbb
Z}^d$, we have ${\mathbf P}[x{\mathcal C}_{\lceil d/ 2\rceil}
y]=1$.
\begin{prop}\label{p.stokdim}
Under ${\mathbb P}$, for any $0\le t_1<t_2<\infty$, the relation
${\mathcal M}_{t_1,t_2}$ has stochastic dimension $2$.
\end{prop}
\begin{proof}
Clearly, it is enough to consider the case $t_1=0$ and
$t_2=u\in(0,\infty)$. First recall that the trajectories in
$\text{supp}(\omega_u)$ that intersect $x\in {\mathbb Z}^d$ can be
sampled in the following way (see for example Theorem 1.1 and
Proposition 1.3 of \cite{sznitmanvacant}):
\begin{enumerate}\label{e.sample}
\item Sample a Poisson random number $N$ with mean $u \text{cap}(x)$
\item Then sample $N$ independent double sided infinite trajectories, where each such trajectory is given by a simple random walk path started at $x$, together with a simple random walk path started at $x$ conditioned on never returning to $x$.
\end{enumerate}

We now establish a lower bound of ${\mathbb P}[x{\CL}_u y]$. Since
any trajectory in $\text{supp}(\omega_u)$ intersecting $x$
contains a simple random walk trajectory started at $x$, we obtain
that
\begin{align}
\BP[x \CL_u y]\stackrel{~\eqref{e.hitbounds}}{\ge} c \la xy
\ra^{-(d-2)}.
\end{align}
Thus the condition ~\eqref{e.sreldef1} is established and it remains to establish the more complicated condition~\eqref{e.sreldef2}.
For this, fix distinct vertices $x,y,z,v\in{\mathbb Z}^d$ and put
$K=\{x,y,z,v\}$. Our next task is to find an upper bound of
$\BP[x\CL_u y,z\CL_u v]$. For $\omega_u=\sum_{i\ge
0}\delta_{w_i^*}$, we let $\hat{\omega}_u=\sum_{i\ge
0}\delta_{w_i^*}\ind_{\{\text{range}(w_i^*)\supset K\}}$. We now
write
\begin{align}\label{e.decompevent}
\BP[x\CL_u y,z\CL_u v]=\BP[x\CL_u y,z\CL_u v,\hat{\omega}_u=0]+\BP[x\CL_u
y,z\CL_u v,\hat{\omega}_u\neq 0],
\end{align}
and deal with the two terms on the right hand side
of~\eqref{e.decompevent} separately. For a point measure
$\tilde{\omega}\le \omega_u$, we write "$x\CL_u y$ in $\tilde{\omega}$"
if there is a trajectory in $\text{supp}(\tilde{\omega})$ whose range contains both $x$ and
$y$. Observe that if $w^*\in \text{supp}(\omega_u-\hat{\omega}_u)$ and $x,y\in \text{range}(w^*)$, then at least one of $z$ or $v$ does not belong to $\text{range}(w^*)$. Hence, the events $\{x\CL_u y\text{ in }\omega_u-\hat{\omega}_u\}$ and $\{z\CL_u v\text{ in }\omega_u-\hat{\omega}_u\}$ are defined in terms of disjoint sets of trajectories, and thus they are independent under ${\mathbb P}$. We get that
\begin{align}\label{e.decompterm1}
\BP[x\CL_u y,z\CL_u v,\hat{\omega}=0]&=\BP[x\CL_u y\text{ in
}\omega_u-\hat{\omega}_u,z\CL_u v\text{ in
}\omega_u-\hat{\omega}_u,\hat{\omega}_u=0]\notag\\ &\le \BP[x\CL_u
y\text{ in
}\omega_u-\hat{\omega}_u,z\CL_u v\text{ in }\omega_u-\hat{\omega}_u]\notag\\
&=\BP[x\CL_u y\text{ in }\omega_u-\hat{\omega}_u]\BP[z\CL_u v\text{ in
}\omega_u-\hat{\omega}_u]\notag\\&\le \BP[x\CL_u y]\BP[z\CL_u v]\notag\\ &\le c(u) (\la x\,y\ra \la z\,v\ra)^{-(d-2)}.
\end{align}
where in the second equality we used the independence that was mentioned above. In addition, we have
\begin{equation}\label{e.decompterm2}
\BP[x\CL_u y,z\CL_u v,\hat{\omega}_u\neq 0]\le
\BP[\hat{\omega}_u\neq 0].
\end{equation}
We now find an upper bound on $\BP[\hat{\omega}_u\neq 0].$ In view of~\eqref{e.decompterm1}, ~\eqref{e.decompterm2} and~\eqref{e.decompevent}, in order to establish~\eqref{e.sreldef2} with $\alpha=d-2$, it is sufficient to show that
\begin{equation}\label{e.sufficient}
\BP[\hat{\omega}_u\neq 0]\le c(u) \la xyzv \ra^{-(d-2)}.
\end{equation}

Using the method of sampling the trajectories from $\omega_u$ containing $x$ and the fact that the law of a simple random walk started at $x$ conditioned on never returning to $x$ is dominated by the law of a simple random walk started at $x$ (here we use that the trajectory of a simple random walk after the last time it visits $x$ has the same distribution as a the trajectory of a simple random walk conditioned on not returning to $x$), we obtain that ${\mathbb P}[\hat{\omega}_u\neq 0]$ is bounded from above by the probability that at least one of $N$ independent double sided simple random walks started at $x$ hits each of $y,z,v$. Here $N$ again is a Poisson random variable with mean $u \text{cap}(x)$.  We obtain that
\begin{align}\label{e.compl}
{\mathbb P}[\hat{\omega}\neq 0]&= 1-\exp(-u\text{cap}(x)
P_x^{\otimes 2}[\{y,z,v\}\subset (X_n')_{n\ge 0}\cup (X_n)_{n\ge
0}])\notag\\ &\le u \text{cap} (x) P_x^{\otimes
2}[\{y,z,v\}\subset (X_n')_{n\ge 0}\cup (X_n)_{n\ge 0}],
\end{align}
where we in the last inequality made use of the inequality $1-\exp(-x)\le x$ for $x\ge 0$.
Here, $P_x^{\otimes 2}[\{y,z,v\}\subset (X_n')_{n\ge 0}\cup (X_n)_{n\ge 0}]$ is the probability that a double sided simple random walk starting at $x$ hits each of $y,z,v$. In order to bound this probability, we first obtain some quite standard hitting estimates. We have
\bae\label{e.supper2}
P_x[H_y<\infty,\,H_z<\infty,\,H_v<\infty]&= \sum_{\stackrel{x_1,x_2,x_3\in} {\text{perm(z,y,v)}}}P_x[H_{x_1}<H_{x_2}<H_{x_3}<\infty]\\ &\le \sum_{\stackrel{x_1,x_2,x_3\in} {\text{perm(z,y,v)}}}P_x[H_{x_1}<\infty]P_{x_1}[H_{x_2}<\infty]P_{x_2}[H_{x_3}<\infty]\\ &\le c \sum_{\stackrel{x_1,x_2,x_3\in} {\text{perm(z,y,v)}}}\left(\la x\,x_1\ra \la x_1\,x_2\ra \la x_2\,x_3\ra\right)^{-(d-2)}\\ &\le c \la K \ra^{-(d-2)},
\eae
where the sums are over all permutations of $z,y,v$. Similarly, for any choice of $x_1$ and $x_2$ from $\{y,z,v\}$ with $x_1\neq x_2$,
\begin{equation}\label{e.supper3}
P_x[H_{x_1}<\infty,\,H_{x_2}<\infty]\le c((\la x\,x_1 \ra \la
x_1\,x_2 \ra)^{-(d-2)}+(\la x\,x_2 \ra \la x_2\,x_1 \ra)^{-(d-2)})
\end{equation}
and
\begin{equation}\label{e.supper4}
P_x[H_{x_1}<\infty]\le c\la x\,x_1 \ra^{-(d-2)}.
\end{equation}

Now set $A=\{\{y,z,v\}\subset (X_n)_{n\ge 0}\}$,
$A'=\{\{y,z,v\}\subset (X_n')_{n\ge 0}\}$,
\begin{equation}\label{B1}B=\bigcup_{t=y,z,v}\{t\subset (X_n)_{n\ge
0},\,K\setminus\{t\}\subset (X_n')_{n\ge 0}\},
\end{equation}
and
\begin{equation}\label{B2}B'=\bigcup_{t=y,z,v}\{t\subset (X_n')_{n\ge
0},\,K\setminus\{t\}\subset (X_n)_{n\ge 0}\}.
\end{equation}

Observe that

\begin{equation}\label{e.eventinclusion}
\{\{y,z,v\}\subset (X_n')_{n\ge 0}\cup (X_n)_{n\ge 0}\}\subset
A\cup A'\cup B \cup B'.
\end{equation}

We have
\begin{equation}\label{e.aestimates}
P_x[A]=P_x[A']\stackrel{~\eqref{e.supper2}}{\le} c\la K \ra
^{-(d-2)}.
\end{equation}
Using the independence between $(X_n)_{n\ge 0}$ and $(X_n')_{n\ge
0}$, it readily follows that
\begin{equation}\label{e.bestimates}
P_x[B]=P_x[B']\stackrel{~\eqref{e.supper3},~\eqref{e.supper4}}{\le}c
\la K \ra^{-(d-2)}.
\end{equation}

From~\eqref{e.eventinclusion}, ~\eqref{e.aestimates} and
~\eqref{e.bestimates} and a union bound, we obtain
\begin{equation}\label{e.penultimateestimate}
P_x^{\otimes 2}[\{y,z,v\}\subset (X_n')_{n\ge 0}\cup (X_n)_{n\ge
0}]\le c \la K \ra^{-(d-2)}.
\end{equation}

Combining~\eqref{e.compl} and~\eqref{e.penultimateestimate}
gives~\eqref{e.sufficient}, finishing the proof of the
proposition.
\end{proof}
\begin{lem}\label{l.lrdim}
The relations $\mathcal{L}$ and $\mathcal{R}$ have stochastic dimension $2$.
\end{lem}
\begin{proof}
We start with the relation ${\mathcal L}$. For $x,y\in {\mathbb
Z}^d$, we have
\begin{align}\label{e.ldim1}
{\bf P}[x{\mathcal L}y]={\bf P}[y\in
\text{range}(w_x)]=P_x[\tilde{H}_y<\infty]
\end{align}
From~\eqref{e.ldim1} and~\eqref{e.hitbounds}, we
obtain
\begin{align}\label{e.ldim3}
c|x-y|^{-(d-2)}\le{\bf P}[x{\mathcal L}y]\le c'|x-y|^{-(d-2)}
\end{align} In addition, for $x,y,z,w\in
{\mathbb Z}^d$, using the independence between the walks $w_x$ and
$w_y$, we get
\begin{align}\label{e.ldim2}
{\bf P}[x{\mathcal L}z,\,y{\mathcal L}w]={\bf P}[x{\mathcal
L}z]{\bf P}[y{\mathcal L}w]\stackrel{~\eqref{e.ldim3}}{\le}
c|x-z|^{-(d-2)} |y-w|^{-(d-2)}.
\end{align}
From~\eqref{e.ldim3} and ~\eqref{e.ldim2} we obtain
$dim_S({\mathcal L})=2$. The proof of the statement
$dim_S({\mathcal R})=2$ is shown by the same arguments.
\end{proof}
Recall the definition of the walks $(w_x)_{x\in {\mathbb Z}^d}$
from above.
\begin{lem}\label{lem:cdimd}
For any $u>0$ and $n\ge 3$, $dim_S(\CC_n)=\min(2n,d)$.
\end{lem}
\begin{proof}
We have \begin{align} \notag dim_S(\CC) & =
dim_S\left(\mathcal{L}\left(\prod_{i=2}^{\left\lceil\frac{d}{2}\right\rceil-1}{\CL}_{u(i-1)/n,ui/n}\right)\mathcal{R}\right)\\\notag
&=\min\left(dim_S({\mathcal
L})+\sum_{i=2}^{\left\lceil\frac{d}{2}\right\rceil-1}dim_S({\CL}_{u(i-1)/n,ui/n})+dim_S({\mathcal
R}),d\right)\\\notag &=\min\left(2+2(\left\lceil
d/2\right\rceil-2)+2,d\right)\\&=\min(2n,d),
\end{align} where we in the second equality used the independence of the
relations and Theorem \ref{thm:dimsum}, and for the third equality used Lemma \ref{l.lrdim} and Proposition \ref{p.stokdim}.
\end{proof}

\section{Tail trivialities}
\begin{definition}
Let $\mathcal{E}$ be a random relation and $v\in\Z^d$. Define the
left tail field corresponding to the vertex $v$ to be
\bae\mathfs{F}^L_{\mathcal{E}}(v)=\bigcap_{K\subset\Z^d\text{
finite}}\sigma\{v\mathcal{E}x:x\notin K\}.\eae We say that
$\mathcal{E}$ is left tail trivial if
$\mathfs{F}^L_{\mathcal{E}}(v)$ is trivial for every $v\in\Z^d$.
\end{definition}
\begin{definition}
Let $\mathcal{E}$ be a random relation and $v\in\Z^d$. Define the
right tail field corresponding to the vertex $v$ be
\bae\mathfs{F}^R_{\mathcal{E}}(v)=\bigcap_{K\subset\Z^d\text{
finite}}\sigma\{x\mathcal{E}v:x\notin K\}.\eae We say that
$\mathcal{E}$ is right tail trivial if
$\mathfs{F}^R_{\mathcal{E}}(v)$ is trivial for every $v\in\Z^d$.
\end{definition}
\begin{definition}
Let $\mathcal{E}$ be a random relation. Define the remote tail
field to be \bae
\mathfs{F}^{Rem}_{\mathcal{E}}=\bigcap_{K_1,K_2\subset\Z^d\text{
finite}}\sigma\{x\mathcal{E}y:x\notin K_1,y\notin K_2\}. \eae We
say that $\mathcal{E}$ is remote tail trivial if
$\mathfs{F}^{Rem}_{\mathcal{E}}$ is trivial.
\end{definition}


\subsection{Left and right tail trivialities}
Recall the definition of the walks $(w_x)_{x\in{\mathbb Z}^d}$
Section~\ref{sec:stochdim}.
\begin{lem}
The relation $\mathcal{L}$ is left tail trivial. The relation
${\mathcal R}$ is right tail trivial.
\end{lem}
\begin{proof}
We start with the relation ${\mathcal L}$. For any $x\in {\mathbb
Z}^d$, we have
\begin{align}\label{e.ltriv1}
{\mathfs F}^L_{{\mathcal
L}}(x)\subset\bigcap_{R>1}\sigma\left\{\text{range}(w_x)\cap
B(x,R)^c\right\}\subset \bigcap_{R>1}\sigma\left\{(w_x(i))_{i\ge
R}\right\}.
\end{align}
Since the $\sigma$-algebra on the right hand side of~\eqref{e.ltriv1} is trivial (\cite{durrett2010probability} Theorem 6.7.5), ${\mathfs F}^L_{{\mathcal L}}(x)$
is trivial for every $x\in {\mathbb Z}^d$. Hence, ${\mathcal L}$
is left tail trivial. Similarly, we obtain that ${\mathcal R}$ is
right tail trivial.
\end{proof}

\subsection{Remote tail triviality}
We omit the details of the following lemma.
\begin{lem}\label{lem:poitotalvariation}
Fix $\mu_0\in {\mathbb R}_+$ and $s\in {\mathbb N}$. For
$\mu>\mu_0$, let $X\sim \text{Pois}(\mu-\mu_0)$ and $Y\sim
\text{Pois}(\mu)$. Then
\begin{equation}\label{e.poistotalvariation}
\sum_{t=0}^{\infty} |P[X=t-s]-P[Y=t]|\to 0\text{ as }\mu\to
\infty.
\end{equation}
\end{lem}
\begin{definition}
For a set $K\subset\Z^d$ denote by
$\eta_K=\omega(W_K^*)=|\{w\in\text{supp}(\omega):w\cap K\neq\phi\}|$.
\end{definition}
\begin{lem}\label{lem:numtotalvar}
Let $K\subset\Z^d$ be a finite set. Denote by $B=B(0,\rho)$, the
ball of radius $\rho$ around $0$. Then for any $s\in {\mathbb N}$,
\[
\sum_{t=0}^\infty\left|\pr[\eta_B=t|\eta_K=s]-\pr[\eta_B=t]\right|\to
0\text{ as }\rho\to\infty.
\]
\end{lem}
\begin{proof}
Write $\eta_B=(\eta_B-\eta_K)+\eta_K.$ Observe that
$\eta_B-\eta_K$ and $\eta_K$ are independent random variables with
distributions $\text{Pois}(u\text{cap}(B)-u\text{cap}(K))$ and
$\text{Pois}(u\text{cap}(K))$ respectively. Consequently
\begin{equation}\label{e.poissoncondequation}
\pr[\eta_B=t|\eta_K=s]=\pr[\eta_B-\eta_K=t-s].
\end{equation}
Since $u\text{cap}(B)\to\infty$ as $\rho\to \infty$, the lemma
follows from~\eqref{e.poissoncondequation} and
Lemma~\ref{lem:numtotalvar}, with the choices
$\mu_0=u\text{cap}(K)$, $\mu=u\text{cap}(B)$, $X=\eta_B-\eta_K$
and $Y=\eta_B$.

\end{proof}

We will need the following lemma, easily deduced from \cite{lawler2010random} Proposition 2.4.2 and Theorem A.4.5. For every $x\in\Z^d$, denote by par$(x)=\sum_{j=1}^{d}x_j$, and even$(x)=\delta_{\text{par}(x)\text{ is even}}$.
\begin{lem}\label{cor:forgetstart}
Let $k>0$, $r>0$, $\epsilon>0$ and $K=B(0,r)\subset\Z^d$. For every \[(x_i,y_i)_{i=1}^k\in\partial K\times\partial K\] we can define $2k$ random walks $(X^i_n)_{i=1}^k,(Y^i_n)_{i=1}^k$, conditioned on never returning to $K$, on the same probability space with initial starting points $X^i_0=x_i,Y^i_0=y_i$ for all $1\leq i\leq k$ such that $((X_n^i)_{n\geq0},(Y_n^i)_{n\geq0})_{i=1}^k$ are independent and there is a $n=n(k,\epsilon,K)>0$ large enough for which
\[
P[\forall 1\leq i\leq k,X^i_m=Y^i_{m+even(x_i-y_i)}\text{ for all }m\geq n]\geq
1-\epsilon .\]
\end{lem}


\begin{lem}
Let $u>0$. The relation $\mathcal{M}_u$ is remote tail trivial.
\end{lem}
\begin{proof}
First we show that it is enough to prove that
$\mathfs{F}^{Rem}_{\mathcal{M}_u}$ is independent of
$\mathfs{F}_K=\sigma\{x\mathcal{M}_uy:x,y\in K\}$ for every finite
$K\subset\Z^d$. So assume this independence. Let
$A\in\mathfs{F}^R_{\mathcal{M}_u}$ and let $K_n$ be finite sets
such that $K_n\subset K_{n+1}$ for every $n$ and $\cup_n
K_n=\Z^d$. Let $M_n=\pr[A|\mathfs{F}_{K_n}]$. Then $M_n$ is a
martingale and $M_n=\pr[A]$ a.s, since we assumed independence.
From Doob's martingale convergence theorem we get that
$M_n\rightarrow\ind_A$ a.s and thus $\pr[A]\in\{0,1\}$.

Let $K\subset \Z^d$ be finite. Suppose
$A\in\mathfs{F}^{Rem}_{\mathcal{M}_u}$, $B\in\mathfs{F}_K$ and
that $\pr[A]>0$, $\pr[B]>0$. According to the above, to obtain the
remote tail triviality of ${\mathcal M}_u$ it is sufficient to
show that for any $\epsilon>0$,
\begin{equation}\label{e.nts1}
|\pr[A|B]-\pr[A]|<\epsilon
.\end{equation}

Let $0<r_1<r_2$ be such that $K\subset B(0,r_1)$. Later, $r_1$ and
$r_2$ will be chosen to be large. Fix $\epsilon>0$. Let
$N=\eta_{B(0,r_1)}$. Let $C=C(K)>0$ and $D=D(r_1)>0$ be so large
that \begin{equation}\label{e.Cchoice}\pr[\eta_K\ge
C]<\epsilon\pr[B]/4\mbox{ and }\pr[N\ge
D]<\epsilon\pr[B]/4.\end{equation} Recall that
$\omega_u|_{W^*_{B(0,r_1)}}=\sum_{i=1}^{N}\delta_{\pi^*(w_i)}$
where $N$ is Pois$(u\text{cap}(B(0,r_1)))$ distributed and
conditioned on $N$, $(w_i(0))_{i=1}^{N}$ are i.i.d. with
distribution $\tilde{e}_{B(0,r_1)}(\cdot)$, $((w_i(k))_{k\ge
0})_{i=1}^{N}$ are independent simple random walks, and
$((w_i(k))_{k\le 0})_{i=1}^{N}$ are independent simple random
walks conditioned on never returning to $B(0,r_1)$ (see for
example Theorem 1.1 and Proposition 1.3 of \cite{sznitmanvacant}).
Letting $\tau_i$ be the last time $(w_i(k))_{k\geq 0}$ visits
$B(0,r_1)$, we have have that $((w_i(k))_{k\ge\tau_i})_{i=1}^{N}$
are independent simple random walks conditioned on never returning
to $B(0,r_1)$. We define the vector
$$\bar{\xi}=\left(w_1(0),\ldots,w_N(0),w_1(\tau_1),\ldots,w_N(\tau_N)\right)\in
\partial (B(0,r_1))^{2N}.$$ Let
$\kappa^+_i=\inf\{k>\tau_i\,:\,w_i(k)\in\partial B(0,r_2)\}$ and
let $\kappa^-_i=\sup\{k<0\,:\,w_i(k)\in\partial B(0,r_2)\}.$
Define the vector
$$\bar{\gamma}=\left(w_1(\kappa^+_1),\ldots,w_N(\kappa^+_N),w_1(\kappa^-_1),\ldots,w_N(\kappa^-_N)\right)\in
\partial (B(0,r_2))^{2N}.$$

Observe that since $A$ belongs to
$\mathfs{F}^{Rem}_{\mathcal{M}_u}$ and
$|\kappa^+_i-\kappa^-_i|<\infty$ for $i=1,...,N$ a.s., we get that
$A$ is determined by $((w_i(k))_{k\ge \kappa^+_i})_{i=1}^N$ and
$((w_i(k))_{k\le \kappa^{-}_i})_{i=1}^N$ and
$\omega_u|_{W^*_{B(0,r_2) ^c}}$. On the other hand, $B$ is
determined by $((w_i(k))_{\kappa^{-}_i \le k \le
\kappa^{+}_i})_{i=1}^N$. In addition, conditioned on $N$ and
$\bar{\gamma}$ we have that $((w_i(k))_{k\ge
\kappa^+_i})_{i=1}^N$, $((w_i(k))_{k\le \kappa^{-}_i})_{i=1}^N$
and $((w_i(k))_{\kappa^{-}_i \le k \le  \kappa^{+}_i})_{i=1}^N$
are conditionally independent. Therefore, conditioned on $N$ and
$\bar{\gamma}$, the events $A$ and $B$ are conditionally
independent. It follows that for any $n\in {\mathbb N}$ and any
$\bar{x}\in (\partial B(0,r_2))^{2n}$

\begin{equation}\label{e.abindep}
\pr[A\cap B|N=n,\bar{\gamma}=\bar{x}]=\pr[A|N=n,\bar{\gamma}=\bar{x}]\pr[B|N=n,\bar{\gamma}=\bar{x}].
\end{equation}

From~\eqref{e.abindep} we easily deduce

\begin{equation}\label{e.abindep2}
\pr[A|B,N=n,\bar{\gamma}=\bar{x}]=\pr[A|N=n,\bar{\gamma}=\bar{x}].
\end{equation}

Therefore,
\bae\label{e.ets1}
|\pr[A|&B]-\pr[A]|=\left|\sum_{n=0}^{\infty}\sum_{\bar{x}\in(\partial B(0,r_2))^{2n}}\pr[A|B,N=n,\bar{\gamma}=\bar{x}]\pr[N=n,\bar{\gamma}=\bar{x}|B]-\pr[A]\right|\\
&\stackrel{~\eqref{e.abindep2}}{=}\left|\sum_{n=0}^{\infty}\sum_{\bar{x}\in(\partial
B(0,r_2))^{2n}}\pr[A|N=n,\bar{\gamma}=\bar{x}]\left[\pr[N=n,\bar{\gamma}=\bar{x}|B]-\pr[N=n,\bar{\gamma}=\bar{x}]\right]\right|\\&\le
\sum_{n=0}^{\infty}\sum_{\bar{x}\in(\partial
B(0,r_2))^{2n}}\left|\pr[N=n,\bar{\gamma}=\bar{x}|B]-\pr[N=n,\bar{\gamma}=\bar{x}]\right|.
\eae Hence, to obtain~\eqref{e.nts1} it will be enough to show
that the double sum appearing in the right hand side
of~\eqref{e.ets1} can be made arbitrarily small by choosing $r_1$
sufficiently large, and then $r_2$ sufficiently large. This will
be done in several steps.


Using Lemma \ref{lem:numtotalvar} we can choose $r_1$ big enough such that for every $m<C$, \begin{equation}\label{e.r1large}\sum_{n=0}^\infty|\pr[N=n|\eta_K=m]-\pr[N=n]|<\epsilon/4C.\end{equation}
Also observe that since $N$ depends only on $\omega_u|_{W^*_{K}}$ through $\eta_K$, we have
\begin{equation}\label{e.Bcondind}
\pr[N=n|B,\eta_K=m]=\pr[N=n|\eta_K=m].
\end{equation}
This gives
\bae\label{eq:Bind}
\sum_{n=0}^\infty|\pr[&N=n|B]-\pr[N=n]|\\
&=\sum_{n=0}^\infty|\sum_{m=0}^\infty\left(\pr[N=n|B,\eta_K=m]-\pr[N=n]\right)\pr[\eta_K=m|B]|\\
&\stackrel{~\eqref{e.Bcondind}}{\le}\sum_{n=0}^\infty\sum_{m=0}^{C-1}\left|\pr[N=n|\eta_K=m]-\pr[N=n]\right|\pr[\eta_K=m|B]\\
&+\sum_{n=0}^\infty\sum_{m=C}^\infty|\pr[N=n|\eta_K=m]-\pr[N=n]|\pr[\eta_K=m|B].\\
\eae

We now estimate the last two lines of ~\eqref{eq:Bind} separately. We have

\begin{equation}\label{e.Bind1}
\sum_{n=0}^\infty\sum_{m=0}^{C-1}\left|\pr[N=n|\eta_K=m]-\pr[N=n]\right|\pr[\eta_K=m|B]\stackrel{~\eqref{e.r1large}}{\le}\epsilon/4 \end{equation}

For the last line of ~\eqref{eq:Bind}, we get
\bae\label{e.Bind2}
\sum_{n=0}^\infty\sum_{m=C}^\infty & |\pr[N=n|\eta_K=m]-\pr[N=n]|\pr[\eta_K=m|B]\\
& \le \sum_{n=0}^\infty\sum_{m=C}^\infty \pr[N=n|\eta_K=m]\pr[\eta_K=m|B]
 + \sum_{n=0}^\infty\sum_{m=C}^\infty \pr[N=n] \pr[\eta_K=m|B]\\
& \le \sum_{n=0}^\infty\sum_{m=C}^\infty
\frac{\pr[N=n,\eta_K=m]}{\pr[B]}+\pr[\eta_K\ge
C|B]\le\frac{\pr[\eta_K\ge C]}{\pr[B]}+\frac{\pr[\eta_K\ge C]}{\pr[B]}\\
&\stackrel{~\eqref{e.Cchoice}}{\le} \frac{\epsilon}{2}. \eae

Combining~\eqref{eq:Bind},~\eqref{e.Bind1} and~\eqref{e.Bind2} we obtain that

\begin{equation}\label{e.Bfinal}
\sum_{n=0}^\infty|\pr[N=n|B]-\pr[N=n]|\le \frac{3\epsilon}{4}.
\end{equation}

By Lemma \ref{cor:forgetstart}, we can choose $r_2>r_1$ large
enough so that for any $n\le D$ and for any $\bar{y}\in(\partial
B(0,r_1))^{2n}$ , \bae\label{e.mixing}
&\sum_{\bar{x}\in\partial B(0,r_2)^n}\left|\pr[\bar{\gamma}=\bar{x}|B,N=n,\bar{\xi}=\bar{y}]-\pr[\bar{\gamma}=\bar{x}|N=n]\right|=\\
&\sum_{\bar{x}\in\partial
B(0,r_2)^n}\left|\pr[\bar{\gamma}=\bar{x}|N=n,\bar{\xi}=\bar{y}]-\pr[\bar{\gamma}=\bar{x}|N=n]\right|<\epsilon/2
,\eae

where the first equality can be shown in a way similar to~\eqref{e.Bcondind}.

Thus, for any $n\le D$,
\bae\label{eq:r3bigenough}
&\sum_{\bar{x}\in\partial B(0,r_2)^{2n}}\left|\pr[\bar{\gamma}=\bar{x}|B,N=n]-\pr[\bar{\gamma}=\bar{x}|N=n]\right|\\
&\le\sum_{\bar{x}\in\partial B(0,r_2)^{2n}}\sum_{\bar{y}\in\partial B(0,r_1)^{2n}}\left|\pr[\bar{\gamma}=\bar{x}|B,N=n,\bar{\xi}=\bar{y}]-\pr[\bar{\gamma}=\bar{x}|N=n]\right|\pr[\bar{\xi}=\bar{y}|B,N=n]\\
&=\sum_{\bar{y}\in\partial B(0,r_1)^{2n}}\sum_{\bar{x}\in\partial B(0,r_2)^{2n}}\left|\pr[\bar{\gamma}=\bar{x}|B,N=n,\bar{\xi}=\bar{y}]-\pr[\bar{\gamma}=\bar{x}|N=n]\right|\pr[\bar{\xi}=\bar{y}|B,N=n]\\
&\stackrel{~\eqref{e.mixing}}{<}\epsilon/2. \eae We now have what
we need to bound~\eqref{e.ets1}. \bae\label{e.longfinal}
\sum_{n=0}^{\infty}&\sum_{\bar{x}\in \partial B(0,r_2)^{2n}}\left|\pr[N=n,\bar{\gamma}=\bar{x}|B]-\pr[N=n,\bar{\gamma}=\bar{x}]\right|\\
&=\sum_{n=0}^{\infty}\sum_{\bar{x}\in \partial B(0,r_2)^{2n}}\left|\pr[\bar{\gamma}=\bar{x}|B,N=n]\pr[N=n|B]-\pr[\bar{\gamma}=\bar{x}|N=n]\pr[N=n]\right|\\
&\leq\sum_{n=0}^{\infty}\sum_{\bar{x}\in \partial B(0,r_2)^{2n}}\left|\pr[\bar{\gamma}=\bar{x}|B,N=n]\pr[N=n|B]-\pr[\bar{\gamma}=\bar{x}|N=n]\pr[N=n|B]\right|\\
&\quad+\sum_{n=0}^{\infty}\sum_{\bar{x}\in \partial B(0,r_2)^{2n}}\left|\pr[\bar{\gamma}=\bar{x}|N=n]\pr[N=n|B]-\pr[\bar{\gamma}=\bar{x}|N=n]\pr[N=n]\right|\\
&\leq\sum_{n=0}^{\infty}\sum_{\bar{x}\in \partial B(0,r_2)^{2n}}\left|\pr[\bar{\gamma}=\bar{x}|B,N=n]-\pr[\bar{\gamma}=\bar{x}|N=n]\right|\pr[N=n|B]\\
&\quad+\sum_{n=0}^{\infty}\sum_{\bar{x}\in \partial B(0,r_2)^{2n}}\pr[\bar{\gamma}=\bar{x}|N=n]\left|\pr[N=n|B]-\pr[N=n]\right|\\
&\stackrel{~\eqref{e.Bfinal}}{\le}\sum_{n=0}^{D}\sum_{\bar{x}\in \partial B(0,r_2)^{2n}}\left|\pr[\bar{\gamma}=\bar{x}|B,N=n]-\pr[\bar{\gamma}=\bar{x}|N=n]\right|\pr[N=n|B]\\
&\quad+\sum_{n=D}^{\infty}\sum_{\bar{x}\in \partial B(0,r_2)^{2n}}\left|\pr[\bar{\gamma}=\bar{x}|B,N=n]-\pr[\bar{\gamma}=\bar{x}|N=n]\right|\pr[N=n|B]+\frac{3\epsilon}{4}\\
&\stackrel{~\eqref{eq:r3bigenough}}{\le}\frac{\epsilon}{2}+2\pr[N\ge
D|B]+\frac{3\epsilon}{4}\le \frac{\epsilon}{2}+2\frac{\pr[N\ge
D]}{\pr[B]}+\frac{3\epsilon}{4}\stackrel{~\eqref{e.Cchoice}}{\le}\frac{7\epsilon}{4}.
\eae where the first inequality follows from the triangle
inequality. Since $\epsilon>0$ is arbitrary, we deduce that
$\pr[A|B]=\pr[A]$ from~\eqref{e.ets1} and~\eqref{e.longfinal}. The
triviality of the sigma algebra ${\mathfs F}_{{\mathcal
M}_u}^{Rem}$ is therefore established.
\end{proof}


\section{Upper bound}\label{sec:upper}
In this section, we provide the proof of the upper bound of
Theorem~\ref{thm:main}. Throughout this section, fix $n=\lceil d/2
\rceil$. Recall the definition of the trajectories $(w_x)_{x\in
{\mathbb Z}^d}$ from Section~\ref{sec:stochdim}. We have proved in
Lemma \ref{lem:cdimd} that the random relation $\mathcal{C}_{n}$
has stochastic dimension $d$, and therefore
$\inf_{x,y\in\Z^d}{\mathbf P}[x\mathcal{C}_{n} y]>0$. Since
$\mathcal{L}$ is left tail trivial, $\mathcal{R}$ is right tail
trivial and the relations $\mathcal{M}_{u(i-1)/n,ui/n}$ are remote
tail trivial for $i=1,...,n$, we obtain from Corollary 3.4 of
\cite{benjamini2004geometry} that
\begin{align}\label{e.punch1} {\mathbf P}[x\mathcal{C}_n y]=1\text{ for every
}x,y\in\Z^d.\end{align} Now fix $x$ and $y$ and let $A_1$ be the
event that $x\in {\mathcal I}^{u/n}$ and $A_2$ be the event that
$y\in {\mathcal I}^{(n-1)u/n,u}$. Put $A=A_1\cap A_2$. We now
use~\eqref{e.punch1} to argue that
\begin{align}\label{e.punch2}
\pr\left[x\left( \prod_{i=1}^{n} {\mathcal
M}_{u(i-1)/n,ui/n}\right) y\bigg|A\right]=1.
\end{align}
To see this, first observe that $A$ is the event that
$\omega_{0,u/n}(W^*_x)\ge 1$ and $\omega_{u(n-1)/n,u}(W^*_y)\ge
1$. Consequently, on $A$, ${\mathcal I}(\omega_{u/n}|_{W_y^*})$
contains at least one trace of a simple random walk started at $x$
and hence stochastically dominates $\text{range}(w_y)$. In the
same way, ${\mathcal
I}(\omega_{u(n-1)/n,u/n}|_{W_x^*})$stochastically dominates
$\text{range}(w_y)$. Thus we obtain
\begin{align}\label{e.punch3}
{\mathbb P}\left[x\left( \prod_{i=1}^{n} {\mathcal
M}_{u(i-1)/n,ui/n}\right) y\bigg|A\right]\ge {\mathbf
P}[x\mathcal{C}_n y]=1,
\end{align}
giving~\eqref{e.punch2}. Equation~\eqref{e.punch2} implies that if
$x\in {\mathcal I}^{u/n}$ and $y\in {\mathcal I}^{(n-1)u/n,u}$,
then $x$ and $y$ are ${\mathbb P}$-a.s. connected in the ranges of
at most $\lceil d/2 \rceil$ trajectories from
$\text{supp}(\omega_u)$.

Now let $I_1=[t_1,t_2]\subset [0,u]$ and $I_2=[t_3,t_4]\subset
[0,u]$ be disjoint intervals. Let $A_{I_1,I_2}$ be the event that
$x\in {\mathcal I}^{t_1,t_2}$ and $y\in {\mathcal I}^{t_3,t_4}$.
The proof of~\eqref{e.punch2} is easily modified to obtain
\begin{equation}\label{e.punch4}
{\mathbb P}\left[x\left( \prod_{i=1}^{n} {\mathcal
M}_{u(i-1)/n,ui/n}\right) y\bigg|A_{I_1,I_2}\right]=1.
\end{equation}

Observe that up to a set of measure $0$, we have

\begin{equation}\label{e.i}
\{x\in {\mathcal I}^u,\,y\in {\mathcal I}^u\}=\{x{\mathcal M}_u
y\}\cup\left(\bigcup_{I_1,I_2} \{x\in {\mathcal
I}^{t_1,t_2},\,y\in {\mathcal I}^{t_3,t_4}\}\right),
\end{equation}

where the union is over all disjoint intervals
$I_1=[t_1,t_2],I_2=[t_3,t_4]\subset [0,u]$ with rational distinct
endpoints. Observe that all the events in the countable union on
the right hand side of~\eqref{e.i} have positive probability. In
addition, due to~\eqref{e.punch4}, conditioned on any of them, we
have $x\left( \prod_{i=1}^{n} {\mathcal M}_{u(i-1)/n,ui/n}\right)
y$ a.s. Therefore, we finally conclude that
\begin{equation}\label{e.punch5}
{\mathbb P}\left[x\left( \prod_{i=1}^{n} {\mathcal
M}_{u(i-1)/n,ui/n}\right) y\bigg|x,y\in {\mathcal I}^u\right]=1,
\end{equation}
finishing the proof of the upper bound of Theorem~\ref{thm:main}.

\section{Lower bound}\label{sec:lower}
In this section, we provide the proof of the lower bound of
Theorem~\ref{thm:main}. More precisely, we show that with
probability one, there are vertices $x$ and $y$ contained in
${\mathcal I}^u$ that are not connected by a path using at most
$\lceil \frac{d}{2} \rceil-1$ trajectories from
$\text{supp}(\omega_u)$. Recall the definition of the
decomposition of $\omega_u$ into $\omega_u^0$, $\omega_u^1$,...
from Section~\ref{sec:randominterlace}. For $k=0,1,...$, define
$$V_k=\bigcup_{w^*\in\text{supp}(\sum_{i=0}^k \omega_u^i)}w^*({\mathbb Z}).$$
 In addition, let
$V_{-1}=\{0\}$ and $V_{-2}=\emptyset$. Observe that with this
notation,
$$\omega_u^k=\omega_u|_{(W^*_{V_{k-1}}\setminus W^*_{V_{k-2}})},\,k=0,1,...$$
Here $\omega_{u|_A}$ denotes $\omega_u$ restricted to the set of
trajectories $A\subset W^*$. We also observe that conditioned on
$\omega_u^0,...,\omega_u^{k-1}$, under ${\mathbb P}$,
\begin{equation}\label{e.pproperty}
\omega_u^k\text{ is a Poisson point process on }W^*\text{ with
intensity measure }u \ind_{W^*_{V_{k-1}}\setminus
W^*_{V_{k-2}}}\nu(dw^*),
\end{equation}
see the Appendix for details. We now construct the vector
$(\bar{\omega}_u^0,...,\bar{\omega}_u^k)$ with the same law as the
vector $(\omega_u^0,...,\omega_u^k)$ in the following way. Suppose
$\sigma_0,\sigma_1,...$ are i.i.d. with the same law as
$\omega_u$. Let $\bar{\omega}_u^0=\sigma_0|_{W^*_{\{0\}}}$ and
then proceed inductively as follows: Given
$\bar{\omega}_u^0,...,\bar{\omega}_u^k$, define
$$\bar{V}_k=\bigcup_{w^*\in\text{supp}(\sum_{i=0}^k \bar{\omega}_u^i)}w^*({\mathbb Z}),$$ and let $\bar{V}_{-1}=\{0\}$ and $\bar{V}_{-2}=\emptyset$.
Then let
$$\bar{\omega}_u^{k+1}=\sigma_{k+1}|_{(W^*_{\bar{V}_{k}}\setminus
W^*_{\bar{V}_{k-1}})}.$$ Using~\eqref{e.pproperty} one checks that
in this procedure, for any $k\ge 0$, the vector
$(\bar{\omega}_u^0,...,\bar{\omega}_u^k)$ has the same law as
$(\omega_u^0,...,\omega_u^k)$.

Let $m=\lceil \frac{d}{2} \rceil-1$. We now get that

\begin{align}\label{e.lboundprel}
{\mathbb P}[0{\mathcal M}_u^{(m)}x] & ={\mathbb
P}\left[0\stackrel{V_{m-1}}{\longleftrightarrow}x\right]={\mathbb
P}^{\otimes
n}\left[0\stackrel{\bar{V}_{m-1}}{\longleftrightarrow}x\right].
\end{align}
The event
$\left\{0\stackrel{\bar{V}_{m-1}}{\longleftrightarrow}x\right\}$
is the event that there is some $l\le m-1$ and trajectories
$\gamma_i\in \bar{\omega}_u^i$, $i=0,...,l$, such that
$\gamma_i({\mathbb Z})\cap \gamma_{i+1}({\mathbb Z})\neq
\emptyset$, $0\in \gamma_0({\mathbb Z})$ and $x\in
\gamma_l({\mathbb Z})$. Since $\bar{\omega}_u^i\le \sigma_i$, we
obtain
\begin{align}\label{e.lboundfinal}
{\mathbb P}^{\otimes
n}\left[0\stackrel{\bar{V}_{m-1}}{\longleftrightarrow}x\right]\le \sum_{l=0}^{m-1}
{\mathbb P}^{\otimes n}\left[0 \prod_{i=0}^{l}\left({\mathcal
M}_{u} (\sigma_i) \right)x\right],
\end{align}
where we use the notation ${\mathcal M}_{u} (\sigma_i)$ for the
random relation defined in the same way as ${\mathcal M}_u$, but
using $\sigma_i$ instead of $\omega_u$. Now use the independence
of the $\sigma_i$'s and the fact that $\text{dim}_{\mathcal
S}({\mathcal M}_{u} (\sigma_i))=2$, to obtain that for any $l\le m-1$,
\begin{equation}\label{e.lboundstochdim}\text{dim}_{\mathcal
S}\left(\prod_{i=0}^{l}\left({\mathcal M}_{u} (\sigma_i)
\right)\right)\le 2 m<d.\end{equation} Therefore by~\eqref{e.lboundprel},~\eqref{e.lboundfinal} and~\eqref{e.lboundstochdim},
\begin{align}\label{e.low1}{\mathbb P}[0{\mathcal
M}_u^{(m)}x]\to 0\text{ as }|x|\to \infty.
\end{align}
Put $\hat{\omega}_u^m=\sum_{i=0}^{m}\omega_u^i$. Observe that
Equation~\eqref{e.low1} can be restated as
\begin{align}\label{e.low2}{\mathbb P}\left[x\in {\mathcal I}(\hat{\omega}_u^{m-1})\right]\to 0\text{ as }|x|\to \infty.\end{align} For $n\ge 1$, let $x_n=n e_1$. For $n\ge 1$, we define
the events $A_n=\{x_n\in {\mathcal I}^u(\hat{\omega}_u^{m-1})\}$
and $B_n=\{x_n\in {\mathcal I}^u\}$ Using~\eqref{e.low2} we can
find a sequence $(n_k)_{k=1}^{\infty}$ such that
\begin{align}\sum_{k=1}^{\infty} {\mathbb P}[A_{n_k}]<\infty.\end{align} By the Borel-Cantelli lemma, \begin{align}\label{e.low3}{\mathbb P}[A_{n_k}\text{i.o.}]=0.\end{align}
On the other hand, by ergodicity (see Theorem 2.1 in
\cite{sznitmanvacant}), we have
\begin{align}\label{e.low4}
\lim_{n\to\infty}\frac{1}{n}\sum_{k=1}^{n}{\bf
1}_{B_{n_k}}={\mathbb P}[0\in {\mathcal I}^u] \text{ a.s.}
\end{align}
Since ${\mathbb P}[0\in {\mathcal I}^u]>0$,
Equation~\eqref{e.low4} implies that
\begin{align}\label{e.low5}
{\mathbb P}[B_{n_k}\text{ i.o.}]=1.
\end{align}
From equations~\eqref{e.low4} and~\eqref{e.low5} it readily
follows that ${\mathbb P}[\cup_{i\ge 1}(B_i\setminus A_i)]=1$,
which means that a.s. the set $\{y\in {\mathbb Z}^d\,:\,y\in
{\mathcal I}^u,\,y\notin {\mathcal I}^u(\hat{\omega}_u^{m-1})\}$
is non-empty. In particular, on the event that $0\in {\mathcal
I}^u$, we can a.s. find points belonging to ${\mathcal I}^u$ that
cannot be reached from $0$ using the ranges of at most $m=\lceil
d/2 \rceil -1$ trajectories from $\text{supp}(\omega_u)$.\qed


\section{Open questions}
The following question was asked by Itai Benjamini: Given two
points $x,y\in\Z^d$, estimate the probability that $x$ and $y$ are
connected by at most $\left\lceil\frac{d}{2}\right\rceil$
trajectories intersecting a ball of radius $r$ around the origin.

Answering the first question can help solve the question of how
one finds the $\left\lceil\frac{d}{2}\right\rceil$ trajectories
connecting two points in an efficient manner.\newline


{\bf Acknowledgments:}

We wish to thank Itai Benjamini for suggesting this problem and
some discussions, and Noam Berger for discussions and feedback.

\section{Appendix}\label{sec:appendix}
Here we show a technical lemma (Lemma~\ref{l.condpoisslemma}
below) needed in the proof of the lower bound in Section
~\ref{sec:lower}. For the proof of Lemma~\ref{l.condpoisslemma},
we need the following lemma, which is standard and we state
without proof.
\begin{lem}\label{l.helplemma}
Let $X$ be a Poisson point process on $W^*$, with intensity
measure $\rho$. Let $A\subset W^*$ be chosen at random
independently of $X$. Then, conditioned on $A$, the point process
$1_A X$ is a Poisson point process on $W^*$ with intensity measure
$1_A \rho$.
\end{lem}

Write $\omega_u=\sum_{k=0}^{\infty}\omega_u^k$ where $\omega_u^k$
is defined in the end of Section ~\ref{sec:randominterlace}. Put
$V_{-2}=\emptyset$ and $V_{-1}=\{0\}$ and
\begin{equation}\label{e.vdef}
V_k={\mathcal I}\left(\sum_{i=0}^k \omega_u^i\right),\,k=0,1,...
\end{equation}
Recall that
\begin{equation}\label{e.omeq1}
\omega_u^k=\omega_u|_{(W^*_{V_{k-1}}\setminus
W^*_{V_{k-2}})},\,k=0,1,...
\end{equation}
Introduce the point process
\begin{equation}\label{e.omeq2}
\tilde{\omega}_u^k=\omega_u|_{W^*\setminus
W^*_{V_{k-1}}},\,k=0,1,...
\end{equation}
For $k\ge 0$, write ${\mathbb P}_k$ for ${\mathbb P}$ conditioned
on $\omega_u^0,..,\omega_u^k$.
\begin{lem}\label{l.condpoisslemma}
Fix $k\ge 0$. Then, conditioned on
$\omega_u^0,...,\omega_u^{k-1}$, the point processes $\omega_u^k$
and $\tilde{\omega}_u^k$ are independent Poisson point processes
on $W^*$, with intensity measures
\begin{equation}\label{e.intmeasure1}
u \ind_{(W^*_{V_{k-1}}\setminus W^*_{V_{k-2}})}\nu(d w^*)
\end{equation}
and
\begin{equation}\label{e.intmeasure2}
u \ind_{(W^*\setminus W^*_{V_{k-1}})}\nu(d w^*),
\end{equation}
respectively.
\end{lem}

\begin{proof}
We will proceed by induction. First consider the case $k=0$. We
have $\omega_u^0=\omega_u|_{W^*_{\{0\}}}$ and
$\tilde{\omega}_u^0=\omega|_{W^*\setminus W^*_{\{0\}}}$. The sets
of trajectories $W^*_{\{0\}}$ and $W^*\setminus W^*_{\{0\}}$ are
non-random. Therefore we get that, using for example Proposition
3.6 in \cite{resnick}, $\omega_u^0$ and $\tilde{\omega}_u^0$ are
Poisson point processes with intensity measures that agree
with~\eqref{e.intmeasure1} and~\eqref{e.intmeasure2} respectively.
In addition, the sets of trajectories $W^*_{\{0\}}$ and
$W^*\setminus W^*_{\{0\}}$ are disjoint, and therefore
$\omega_u^0$ and $\tilde{\omega}_u^0$ are independent.

Now fix some $k\ge 0$ and assume that the assertion of the lemma
is true for $k$. Observe that we have
\begin{equation}\label{e.induct1}
\omega_u^{k+1}=\tilde{\omega}_u^k|_{W^*_{{\mathcal
I}(\omega_u^k)}}
\end{equation}
and
\begin{equation}\label{e.induct2}
\tilde{\omega}_u^{k+1}=\tilde{\omega}_u^k|_{W^*\setminus
W^*_{{\mathcal I}(\omega_u^k)}}.
\end{equation}
By the induction assumption, $\omega_u^k$ and $\tilde{\omega}_u^k$
are independent Poisson process under ${\mathbb P}_{k-1}$. In
particular, under ${\mathbb P}_{k-1}$, $\tilde{\omega}_u^k$ and
$W^*_{{\mathcal I}(\omega_u^k)}$ are independent. Therefore, using
Lemma \ref{l.helplemma} and~\eqref{e.induct1}, we see that if we
further condition on $\omega_u^k$, the point process
$\omega_u^{k+1}$ is a Poisson point process on $W^*$ with
intensity measure given by $u \ind_{W^*_{{\mathcal I}(\omega_u^k)}}
\ind_{(W^*\setminus W^*_{V_{k-1}})}\nu(d w^*)$. However,
\begin{equation}\label{e.int1}
u \ind_{W^*_{{\mathcal I}(\omega_u^k)}} \ind_{(W^*\setminus
W^*_{V_{k-1}})}\nu(d w^*)=u \ind_{(W^*_{V_k}\setminus
W^*_{V_{k-1}})}\nu(d w^*),
\end{equation}
and therefore the claim regarding $\omega_u^{k+1}$ established.
The claim regarding $\tilde{\omega}_u^{k+1}$ follows similarly, by
noting that
\begin{equation}\label{e.int2}
u \ind_{W^*\setminus W^*_{{\mathcal I}(\omega_u^k)}} \ind_{(W^*\setminus
W^*_{V_{k-1}})}\nu(d w^*)=u \ind_{(W^*\setminus W^*_{V_{k}})}\nu(d
w^*).
\end{equation}

\end{proof}

\bibliography{RL}
\bibliographystyle{alpha}
\end{document}